\documentclass[microtype]{gtpart}
\usepackage[top=1in, left=1.5in, right=1.25in, bottom=2in]{geometry}
\pdfoutput=1

\usepackage{enumerate}
\usepackage[mathscr]{eucal}

\usepackage{overpic}
\usepackage[english]{babel}

\usepackage{hyperref}  
\hypersetup{%
  bookmarksnumbered=true,%
  bookmarks=true,%
  colorlinks=black,%
  linkcolor=black,%
  citecolor=black,%
  filecolor=black,%
  menucolor=black,%
  pagecolor=black,%
  urlcolor=black,%
  pdfnewwindow=true,%
  pdfstartview=FitBH}

\newcommand{\ov}[1]{\overline{#1}}

\newcommand{\bh}{\mathbb H}
\newcommand{\bU}{\mathbb U}

\newcommand{\bB}{\mathbb B}

\newcommand{\bS}{\mathbb S}
\newcommand{\br}{\mathbb R}
\newcommand{\cO}{\mathcal O}

\newcommand{\la}{\langle}
\newcommand{\ra}{\rangle}
\newcommand{\wt}{\widetilde}
\newcommand{\iso}{\cong}
\newcommand{\bndry}{\partial}

\newcommand{\al}{\alpha}

\newcommand{\vp}{\varphi}

\DeclareMathOperator{\Isom}{Isom}
\DeclareMathOperator{\Vol}{Vol}
\DeclareMathOperator{\UT}{T_1\!}

\DeclareMathOperator{\sech}{sech}

\def\co{\colon\thinspace}

\usepackage{thmtools, thm-restate}

\newtheorem{thm}{Theorem}[section]
\newtheorem{Prop}[thm]{Proposition}
\newtheorem{Lem}[thm]{Lemma}

\theoremstyle{remark}
\newtheorem{Rem}[thm]{Remark}

\title{The Bridgeman-Kahn identity for hyperbolic manifolds with cusped boundary}

\author{Nicholas G.~Vlamis}
\givenname{Nicholas}
\surname{Vlamis}
\address{Department of Mathematics \\ University of Michigan \\ 530 Church St. \\ Ann Arbor, MI 48109}
\email{vlamis@umich.edu}
\urladdr{http://www.umich.edu/~vlamis}

\author{Andrew Yarmola}
\givenname{Andrew}
\surname{Yarmola}
\address{Department of Mathematics \\ Boston College \\ 140 Commonwealth Ave. \\ Chestnut Hill, MA 02467}
\email{andrew.yarmola@bc.edu}
\urladdr{https://www2.bc.edu/andrew-v-yarmola}

\begin{document}

\begin{abstract} 
In this note, we extend the Bridgeman-Kahn identity to all finite-volume orientable hyperbolic $n$-manifolds with totally geodesic boundary. In the compact case, Bridgeman and Kahn are able to express the manifold's volume as the sum of a function over only the orthospectrum. For manifolds with non-compact boundary, our extension adds terms corresponding to intrinsic invariants of boundary cusps.
\end{abstract}

\maketitle

\vspace{-3ex}

\section{Introduction}

Let $M$ be a oriented finite-volume  hyperbolic $n$-manifold with nonempty totally geodesic boundary.  An \emph{orthogeodesic} is a (nonoriented) geodesic arc perpendicular to $\partial M$ at both ends.  
The collection of all orthogeodesics on $M$ is denoted $\cO_M$.
The \emph{orthospecturm}, denoted $|\cO_M|$, is the multiset of lengths of elements of $\cO_M$ counted with multiplicity.  When $M$ is compact, Bridgeman-Kahn \cite{Bridgeman:2010fq} show
\[
\mathrm{Vol}(M) = \sum_{\ell\in|\cO_M|} F_n(\ell),
\]
where $\mathrm{Vol}(M)$ is the hyperbolic volume of $M$ and $F_n\co \br_+\to \br_+$ is a decreasing function expressible as an integral of an elementary function.  We will refer to $F_n$ as the $n^{th}$-Bridgeman-Kahn function.

In dimension 2, Bridgeman in \cite{Bridgeman:2011ffa} gives an explicit formula for $F_2$ and also extends the identity to all finite-area orientable hyperbolic surfaces with totally geodesic boundary.  
His work yields the following beautiful identity:
Let $S$ be an oriented finite-area hyperbolic surface with nonempty totally geodesic boundary and $m$ boundary cusps, then
\[
\mathrm{Area}(S) = \frac \pi3m + \sum_{\ell\in |\cO_S|} \frac2\pi \mathcal{L}\left(\sech^2\frac\ell2\right),
\]
where $\mathcal{L}(z)$ is the Rogers dilogarithm.
By applying this identity to ideal polygons in $\bh^2$, Bridgeman was able to give purely geometric proofs of classical functional equations for the Rogers dilogarithm and provide infinite families of new ones.

Masai and McShane \cite{Masai:2013ji} using an integral formula of Calagari \cite{Calegari:2010eh} were able to give a closed form for $F_3$:
\[
F_3(\ell) = \frac{2\pi (\ell+1)}{e^{2\ell}-1}.
\]

As pointed out above, the Bridgeman-Kahn identity extends to non-compact finite-area hyperbolic surfaces.  The purpose of this note is to extend the identity to non-compact finite-volume hyperbolic $n$-manifolds with totally geodesic boundary for $n\geq 3$.

\begin{restatable}{Thm}{bkifv}\label{thm:bkfv} For $n \geq 3$ and $M$ an oriented finite-volume hyperbolic $n$-manifold with nonempty totally geodesic boundary, let $\mathfrak{B}$ be the set of boundary cusps of $M$ and let $|\cO(M)|$ be the orthospectrum. For every $\mathfrak{c} \in \mathfrak{B}$, let $B_\mathfrak{c}$ be an embedded horoball neighborhood of $\mathfrak{c}$ in $M$ and let $d_\mathfrak{c}$ be the Euclidean distance along $\partial B_\mathfrak{c}$ between the two boundary components of $\mathfrak{c}$, then 
\[
\Vol(M) = \sum_{\ell \in |\cO(M)|} F_n(\ell) + \frac{H(n-2) \; \Gamma\left(\frac{n-2}{2}\right)}{\sqrt{\pi} \; \Gamma\left(\frac{n-1}{2}\right)} \sum_{\mathfrak{c} \in \mathfrak{B}}  \frac{\Vol(B_\mathfrak{c})}{d_\mathfrak{c}^{\,n-1}}
\]
where $F_n$ is the $n^{th}$-Bridgeman-Kahn function, $\Vol$ is the hyperbolic volume, $\Gamma(\cdot)$ is the gamma function, and $H(m)$ is the $m^{th}$ harmonic number.
\end{restatable}

The idea of the proof of this identity -- as well as the original Bridgeman-Kahn identity -- is to give a full measure decomposition of the unit tangent bundle of a manifold into pieces indexed by the orthogeodesics and boundary cusps.  
Finding the volume of the pieces indexed by orthogeodesics is the main content of \cite{Bridgeman:2010fq}.  Here, after describing the decomposition, the main content is calculating the volume of the pieces associated to boundary cusps.  
In \S\ref{apollonian} we give an example of a manifold with boundary cusps and calculate its cusp invariants with help from SnapPy.

The standard reference for coordinate versions of the volume form on the unit tangent bundle is Theorem 8.1.1 in a classic of Nicholls \cite{Nicholls:1989ij}. However, there is an error in the calculation and the formula given is off by a factor of $2^{n-2}$.  We record the corrected version here:

\begin{restatable}{Thm}{thmvol}\label{thm:volume} 
Let $dV$ be the hyperbolic volume form on $\bh^{n+1}$ and let $d\omega$ be the spherical volume form on $\bS^n$.
Let $d\Omega = dV d\omega$ be the standard volume form on $\UT \bh^{n+1}.$
\begin{itemize}
\item[(a)]
In the upper half space model of $\bh^{n+1}$, $d\Omega$ is given by
\[
d \Omega = \frac{2^n d \mathbf{x} \, d \mathbf{y} \, d t}{| \mathbf{x} - \mathbf{y}|^{2n}},
\]
where $T_1\bh^{n+1}$ is equipped with the geodesic endpoint parameterization
\[
\UT\bh^{n+1} \iso \{ (\mathbf{x}, \mathbf{y}, t) \in \hat{\br}^{n} \times \hat{\br}^{n} \times \br : \mathbf{x} \neq \mathbf{y}\}.
\]

\item[(b)]
In the conformal ball model of $\bh^{n+1}$, $d\Omega$ is given by
\[
d \Omega = \frac{2^n d \omega(\mathbf{p}) \, d \omega(\mathbf{p}) \, d t}{| \mathbf{p} - \mathbf{q}|^{2n}},
\]
where $T_1\bh^{n+1}$ is equipped with the geodesic endpoint parameterization 
\[
\UT\bh^{n+1} \iso \{ (\mathbf{p}, \mathbf{q}, t) \in \bS^n \times \bS^n \times \br : \mathbf{p} \neq \mathbf{q}\}.
\]
\end{itemize}
\end{restatable}

This correction affects the asymptotincs for the $n^{th}$-Bridgeman-Kahn function as stated in \cite{Bridgeman:2010fq}.  We provide the necessary adjustments here.

\begin{Lem}{\rm (\cite[Lemma 9]{Bridgeman:2010fq})}\label{thm:bko} 
Let $F_n$ be the $n^{th}$ Bridgeman-Kahn function, then
\begin{enumerate}
\item there exists $D_n > 0$, depending only on $n$, such that $$F_n(l) \leq \frac{D_n}{(e^l - 1)^{n-2}}$$
\item $$\lim_{l\to 0} l^{n-2}\, F_n(l)  = \frac{\pi^{\frac{n-2}{2}} \, H(n-2)\, \Gamma\left(\frac{n-2}{2}\right)}{\Gamma\left(\frac{n-1}{2}\right) \, \Gamma\left(\frac{n+1}{2}\right)}$$
\item $$\lim_{l\to \infty} \frac{e^{(n-1)l}}{l}\, F_n(l) = \frac{2^{n-1} \, \pi^{\frac{n-2}{2}} \Gamma\left(\frac{n}{2}\right)}{\Gamma\left(\frac{n+1}{2}\right)^2}$$
\end{enumerate}
\end{Lem}

The proof of Theorem \ref{thm:volume} relies on careful calculations of ball volumes in the different parameterizations of the unit tangent bundle. We do not include them here. We refer the interested reader to the second author's thesis \cite[Chapter 5, \S3]{YarmolaThesis}.

\subsection*{Acknowledgements.}
The authors thank Martin Bridgeman for suggesting the problem.  The first author was supported in part by NSF RTG grant 1045119.


\section{Background on hyperbolic manifolds}

In this note we will use the conformal ball and upper half space models for hyperbolic $n$-space $\bh^n$. A reference for this subsection is \cite{Ratcliffe:2013gs}. Throughout, let $\partial \bh^n$ denote the boundary at infinity of hyperbolic space, $ds$ the length element and $dV$ the volume element of $\bh^n$. The norm $| \cdot |$ will \emph{always} denote the standard Euclidean norm $|\mathbf{x}| = \sqrt{x_1^2 + \ldots + x_n^{2}}$ for $\mathbf{x} \in \br^n$ and $\{\mathbf{e}_i\}_{i = 1}^n$ will be the standard basis for $\br^n.$

Recall that in the \emph{conformal ball model}, one has
$$\bh^{n} \iso  \{ \mathbf{x} \in \br^{n} : |\mathbf{x}| < 1\} = \bB^n, \quad \partial \bh^n \iso \{ \mathbf{x} \in \br^{n} : |\mathbf{x}| = 1\} = \bS^{n-1},$$ $$ds = \frac{2 |d \mathbf{x}|}{1 - | \mathbf{x} |^2} , \quad \text{and} \quad dV = \frac{2^n d \mathbf{x}}{(1 - | x |^2)^n}.$$
Complete geodesics are realized as circular arcs perpendicular to $\bS^{n-1}$ and a hyperplane is the intersection of $\bB^n$ with an $(n-1)$-sphere perpendicular to $\bS^{n-1}$.

In the \emph{upper half space model}, one has
$$\bh^{n} \iso  \{ \mathbf{x} \in \br^{n} : x_n > 0\} = \bU^n, \quad \partial \bh^n \iso \{ \mathbf{x} \in \br^n : x_n = 0\} \cup \{\infty\} = \hat{\br}^{n-1}$$ $$ds = \frac{|d \mathbf{x}|}{x_n} , \quad \text{and} \quad dV = \frac{d \mathbf{x}}{(x_n)^n}.$$
Similarly, complete geodesics are circular arcs or lines perpendicular to $\hat{\br}^{n-1}$ and a hyperplane is the intersection of $\bU^n$ with an $(n-1)$-sphere or a Euclidean hyperplane perpendicular to $\hat{\br}^{n-1}$.

A \emph{halfspace} is the closure of a connected component of $\bh^n$ cut by a hyperplane. A \emph{horoball} is a Euclidean ball tangent to $\bndry \bh^n$ and contained in $\bh^n$ in either of these models. In the upper half space model, a horoball can also be realized as $\{ \mathbf{x} \in \br^{n} : x_n > a > 0\}$. The boundary of a horoball is called a \emph{horosphere} and is Euclidean in the induced path metric from $\bh^n$.

A \emph{hyperbolic $n$-manifold with totally geodesic boundary} $M$ can be defined as an \emph{orientable} manifold with boundary that admits an atlas of charts $\{\vp_\al : U_\al \to D_\al\}$, where $D_\al \subset \bh^n$ are closed halfspaces, $\vp_\al(U_\al \cap \partial M) = \vp_\al(U_\al) \cap \partial D_\al$, and the transition maps are restrictions of elements of $\Isom^+(\bh^n)$. We will assume that all our manifolds are \emph{complete}, in the sense that the developing map $\mathcal{D}\co \wt{M} \to \bh^n$ is a covering map onto the hyperbolic convex hull of some subset of $\bndry \bh^n$. If fact, when $M$ has finite volume, it can be shown that $\mathcal{D}$ is an isometry and $\mathcal{D}(\wt{M})$ is a countable intersection of closed half-spaces bounded by mutually disjoint hyperplanes. Further, if $\Gamma$ is the image of the holonomy map, then $M \iso \text{CH}(\Lambda_\Gamma)/\Gamma$, where $\Lambda_\Gamma = \ov{\Gamma \cdot x} \cap \partial \bh^n$ for any $x \in \bh^n$ is the \emph{limit set} of $\Gamma$ and $\text{CH}(\cdot)$ denotes the hyperbolic convex hull (see \cite{Thurston:1991vg}).

In \cite{Kojima:1990uq}, Kojima shows that whenever $M$ is a complete finite volume hyperbolic $n$-manifold with totally geodesic boundary and $n \geq 3$, then $\partial M$ is a complete finite volume hyperbolic $(n-1)$-manifold. 
In particular, if $X \subset \partial M$ is a boundary component, then $\wt{X}$ is a hyperplane in $\partial \wt{M}$ by completeness.

\subsection{Cusps}\label{hc}  Let $M$ be a complete finite-volume hyperbolic $n$-manifold with or without boundary. Let $\Gamma \leq \Isom^+(\bh^n)$ denote the image of the holonomy map for $M$ and $\wt{M}$ denote the the image of the developing map for $M$. Recall that $\gamma \in \Gamma$ is \emph{parabolic} if and only if it has a unique fixed point on $\bndry \bh^n$. A subgroup of $\Gamma$ is called parabolic if all non-identity elements are parabolic. Let $\mathfrak{C}$ denote the set of conjugacy classes of \emph{maximal} parabolic subgroups of $\Gamma$. The elements of $\mathfrak{C}$ are called \emph{cusps} of $M$. 

One can realize cusps as geometric pieces of $M$. Fix a representative $\Gamma_\mathfrak{c}$ of a cusp $\mathfrak{c} \in \mathfrak{C}$ and recall that all non-identity $\gamma \in \Gamma_\mathfrak{c}$ have the same unique fixed point $x \in \partial\bh^n$. By the Margulis Lemma, there exists some horoball $\wt{B} \subset \bh^n$ centered at $x$ such that $B_\mathfrak{c} = (\wt{B} \cap \wt{M} )/\Gamma_\mathfrak{c}$ embeds into $M$. The piece $B_\mathfrak{c} \subset M$ is called an embedded horoball neighborhood of $\mathfrak{c}$  (see \cite{Ratcliffe:2013gs}). 

Following Kojima \cite{Kojima:1990uq}, $\mathfrak{c}$ arises in two different ways. We say $\mathfrak{c}$ is an \emph{internal} cusp of $M$ whenever $B_\mathfrak{c} \iso E_\mathfrak{c} \times [0,\infty)$ for some closed Euclidean $(n-1)$-manifold $E_\mathfrak{c}$. We call $\mathfrak{c}$ a \emph{boundary cusp}, or $\partial$-cusp for short, whenever $B_\mathfrak{c} \iso E_\mathfrak{c}^\partial \times [0,\infty)$ for some compact Euclidean $(n-1)$-manifold $E_\mathfrak{c}^\partial$ with totally geodesic boundary.  In the case of a $\partial$-cusp, $\partial E_\mathfrak{c}^\partial$ is composed to two \emph{parallel} components which correspond to horoball neighborhoods of cusps of $\partial M$. In particular, a $\partial$-cusp gives rise to a pair of cusps in $\partial M$. In the universal cover, a $\partial$-cusp corresponds to a point of tangency between two hyperplanes in $\partial \wt{M}$. The set of boundary cusps of $M$ will be denoted $\mathfrak{B}$.

\subsection{Unit tangent bundle}

Let $dV$ denote the hyperbolic volume form on $\bh^n$ and let $d\omega$ be the volume element on $\bS^{n-1}$ induced from the Euclidean volume form on $\br^n$ with $$\Vol(\bS^{n-1}) = \frac{2\, \pi^{n/2}}{\Gamma(n/2)}.$$  The natural volume element on the unit tangent bundle $T_1\bh^n$ of $\bh^n$ is given by $d\Omega = dV\, d\omega$. Here, natural means that $d\Omega$ is invariant by the action of the group of hyperbolic isometries $\mathrm{Isom}(\bh^n)$.

For computations, we will use the geodesic endpoint parameterization for $T_1\bh^{n}$ defined as follows. Fix a base point $O  \in \bh^n$. For convenience, we will choose the origin $\mathbf{0} \in \br^n$ in the conformal ball model and $\mathbf{e}_n \in \bU^n$. In the \emph{geodesic endpoint parametrization} a point $v \in \UT\bh^n$ is mapped to a triple $(\xi_-,\xi_+,t) \in \bndry \bh^n \times \bndry \bh^n \times \br$ where $\xi_-, \xi_+$ are the backwards and forwards endpoints of the geodesic defined by $v$, respectively. On this geodesic there is a closest point $p(\xi_-, \xi_+)$ to $O$, called the \emph{reference point}. The value of $t$ is the signed hyperbolic distance along this geodesic from $p(\xi_-, \xi_+)$ to the basepoint of $v$. This assignment gives a bijection  $$\UT\bh^{n} \iso \{ (\xi_-, \xi_+, t) \in \bndry \bh^n \times \bndry \bh^n \times \br : \xi_- \neq \xi_+\}.$$
We express $d\Omega$ with respect to this parametrization using Theorem  \ref{thm:volume} as stated in the introduction.


\section{An Apollonian manifold}
\label{apollonian}

\begin{figure}[thb]
  \centering
  \begin{minipage}[b]{0.32\textwidth}
    \includegraphics[width=\textwidth]{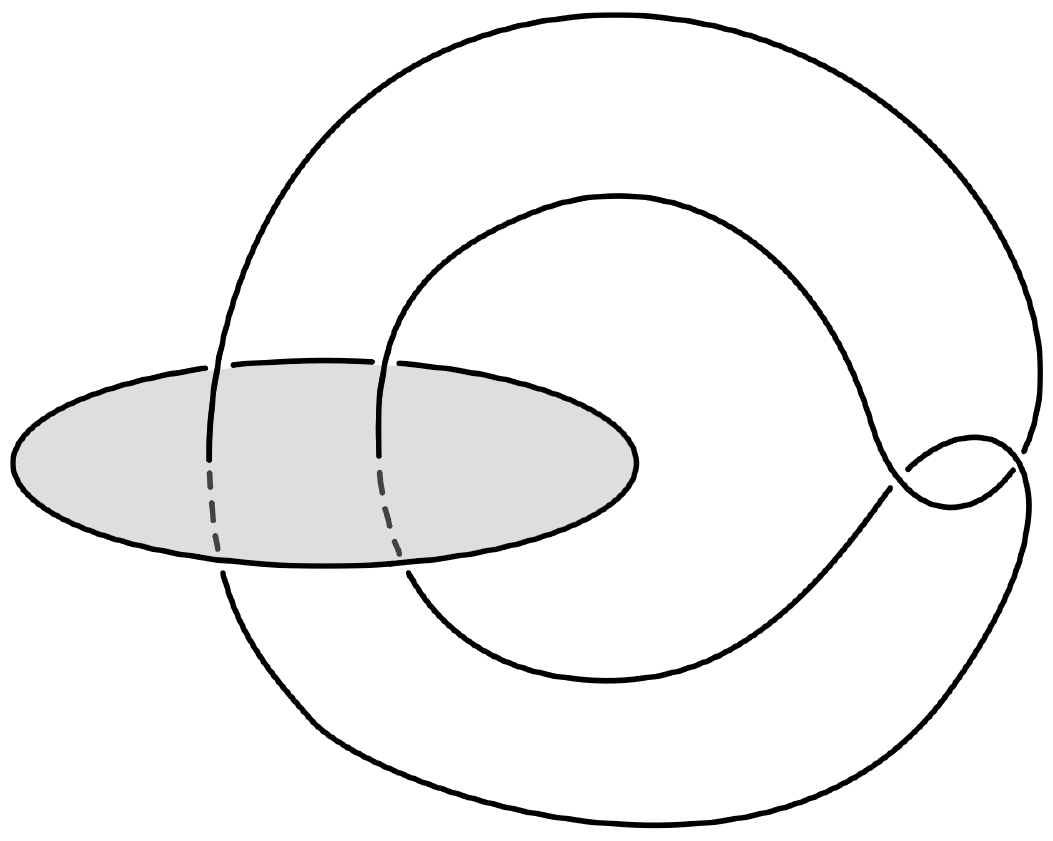}
    \caption{Whitehead link}
    \label{fig:whl}
  \end{minipage}
  \hspace*{6ex}
  \begin{minipage}[b]{0.49\textwidth}
    \includegraphics[width=\textwidth]{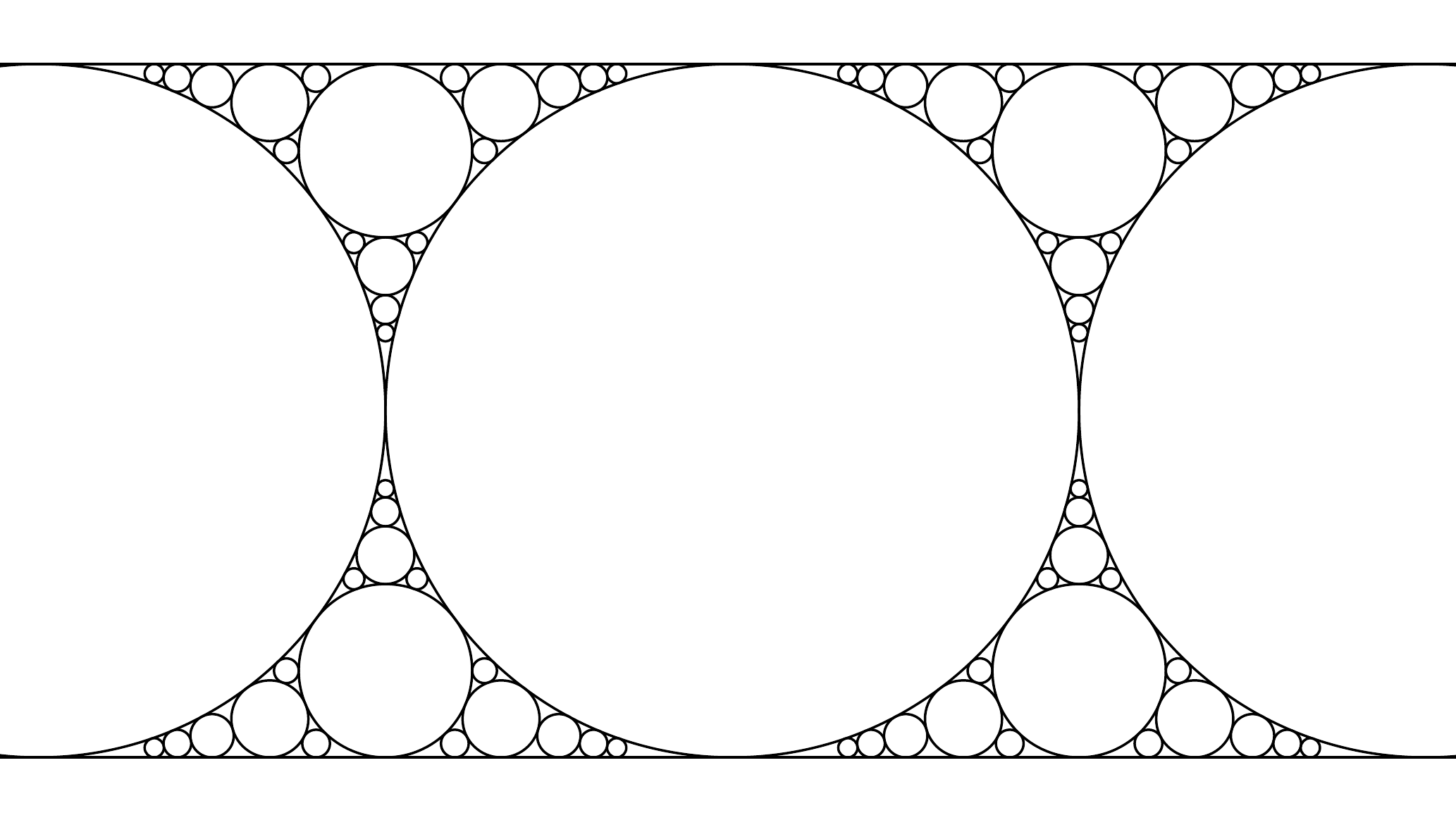}
    \caption{A piece of the Apollonian strip in $\partial \bU^3$}
    \label{fig:apl}
  \end{minipage}
\end{figure}

Before moving on to the proof of Theorem \ref{thm:bkfv}, we explore an example of a 3-manifold $M$ with $\partial$-cusps and compute the cusp invariants in the extended Bridgeman-Kahn identity. In his course notes on Riemann surfaces, dynamics and geometry, McMullen remarks on the following wonderful construction \cite[\S 6.15]{McMullen:itQmnWXs} (see also \cite[\S2]{OhApollonian} and \cite[Chapter 7]{MumfordIndra} for additional details). To begin, we take the complement in $\bS^3$ of the Whitehead link, shown in Figure \ref{fig:whl}. The complement admits a unique complete hyperbolic structure $N$ with two internal cusps. The shaded surface in Figure \ref{fig:whl} is isotopic,  relative to the torus boundary components, to a totally geodesic thrice punctured sphere $\Sigma$ in $N$. Cutting along $\Sigma$, one obtains a hyperbolic manifold $M$ with totally geodesic boundary with three $\partial$-cusps. Fascinatingly, the limit set of the holonomy representation for $M$ is conjugate to the Apollonian strip in $\partial \bU^3$ shown in Figure \ref{fig:apl}.

\begin{figure}[hbt]
    \begin{center}
\begin{overpic}[scale=.3]{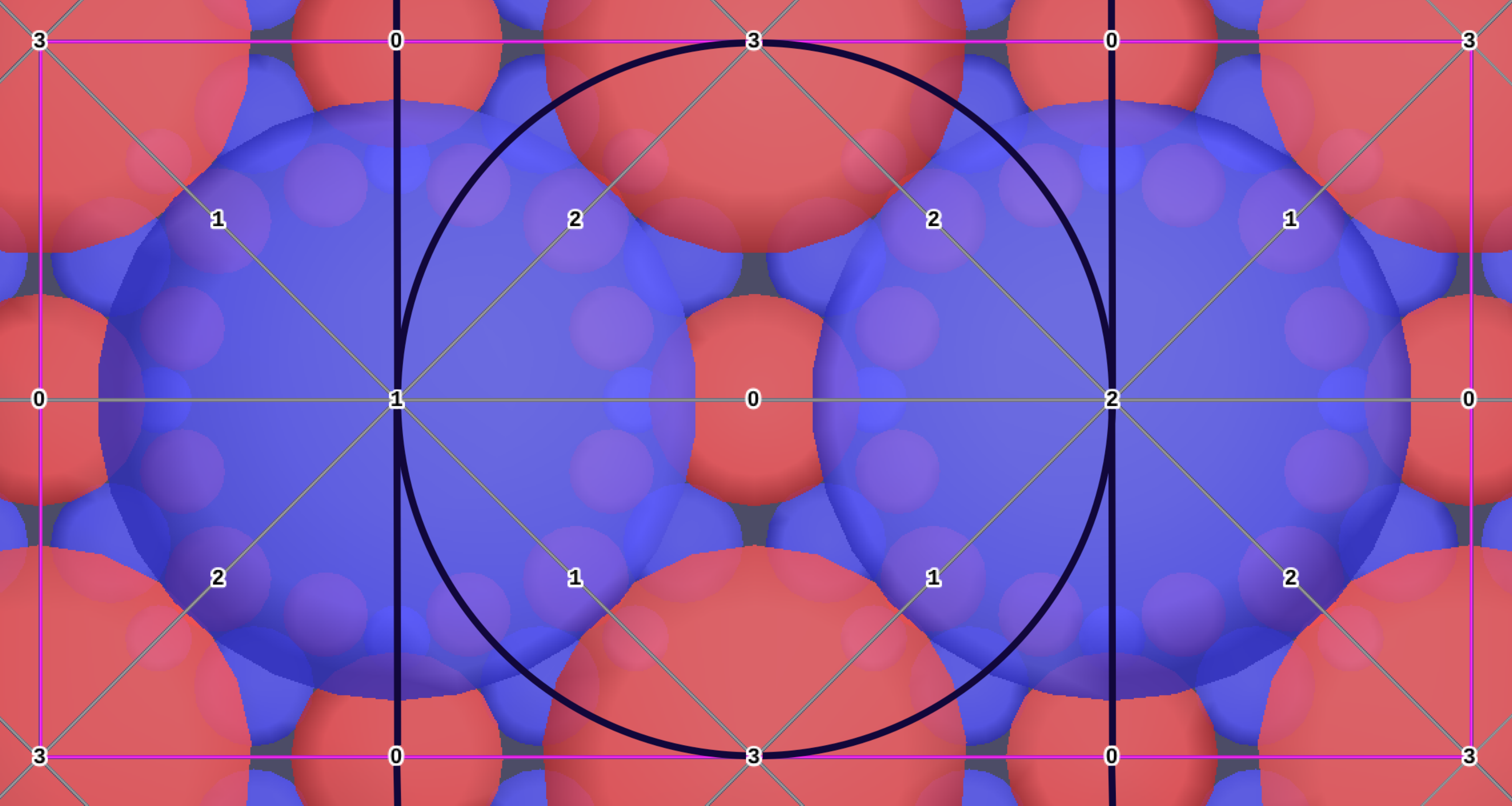}
\put(48.7,46){$A$}
\put(48.7,6){$B$}
\put(75,25.8){$C$}
\put(22,25.8){$D$}
\end{overpic}
\end{center}
\caption{$\wt{B}_\mathfrak{r} \cup \wt{B}_\mathfrak{b}$ wth a red (light) cusp at infinity.}
\label{fig:horo_red}
\end{figure}
\begin{figure}[hbt]
 
\begin{center}
\begin{overpic}[scale=.3]{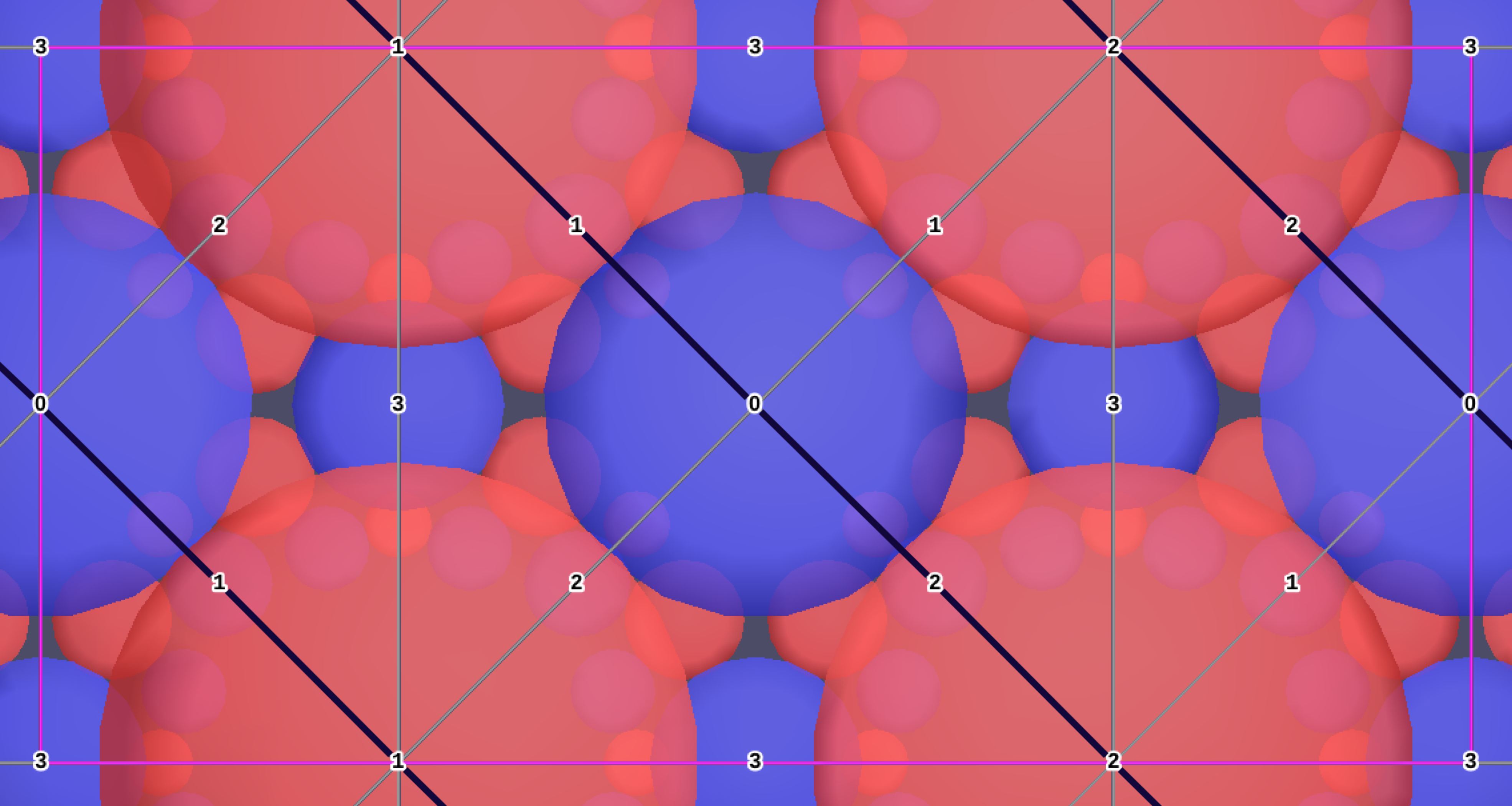}
\put(23,47){$B$}
\put(75,4.5){$A$}
\put(52,25.5){$C$}
\end{overpic}
\end{center}
\caption{$\wt{B}_\mathfrak{r} \cup \wt{B}_\mathfrak{b}$ with a blue (dark) cusp at infinity.}
\label{fig:horo_blue}

\end{figure}

Using SnapPy \cite{SnapPy}, we can get a geometric picture of $N$ and find the invariants associated to the boundary cusps of $M$. Let $\{\mathfrak{r}, \mathfrak{b} \}$ be the cusps of $N$. We will call $\mathfrak{r}$ the red (or light) cusp and $\mathfrak{b}$ the blue (or dark) cusp. Note that there is an isometry $\rho$ of $N$ exchanging $\mathfrak{r}, \mathfrak{b}$, so we are not making any hidden choices. In fact, we may choose $\rho$-symmetric embedded horoball neighborhoods $B_\mathfrak{r}, B_\mathfrak{b}$ of $\mathfrak{r}, \mathfrak{b}$ in $N$ that are jointly maximal. Each $B_\mathfrak{r}$ and $B_\mathfrak{b}$ has volume $\sqrt{2}$. Figure \ref{fig:horo_red} is a diagram of $\wt{B}_\mathfrak{r} \cup \wt{B}_\mathfrak{b}$ as viewed from the perspective of a red (light) horoball at infinity and Figure \ref{fig:horo_blue} is the perspective form the blue (dark) horoball at infinity. The highlighted rectangles are fundamental domains of the corresponding maximal parabolic subgroups fixing infinity. The rectangles have sides of Euclidean length $\sqrt[4]{2}$ and $2 \sqrt[4]{2}$ on both $\partial B_\mathfrak{r}$ and $\partial B_\mathfrak{b}$. The numbered edges in Figure \ref{fig:horo_red} and Figure \ref{fig:horo_blue}  are hyperbolic geodesics connecting ideal points of horoballs. These edges correspond to the lift of an ideal triangulation of $N$ and edges with the same number are in the same $\pi_1(N)$-orbit. Note that there are edges that go out of the page and up to the horoball at infinity. 

In Figure \ref{fig:horo_red} we find four horoballs with ideal points $A,B,C,D$. These points form a Euclidean square and, in particular, there is a unique hyperbolic plane $P$ containing these points. It follows that $A,B,C,D$ are the ideal points of an ideal quadrilateral $Q$ contained in $P$. Furthermore, $Q$ projects to the cutting thrice-punctured sphere $\Sigma$ as can be seen by the edge identifications.

The edge labels tell us that there is a parabolic isometry $\gamma_A \in \pi_1(N)$ taking the geodesic with end points $\{C,A\}$ to the geodesic with endpoints $\{C,\infty\}$. Since the derivative of a parabolic acting on $\mathbb{C}$ at its fixed point is 1, we conclude that $P_A = \gamma_A(P)$ is perpendicular to the page with the vertical line through $C$ as its boundary. Similarly, there is a parabolic isometry $\gamma_B \in \pi_1(N)$ taking the geodesic with end points $\{D,B\}$ to the geodesic with endpoints $\{D,\infty\}$ and $P_B = \gamma_B(P)$ is perpendicular to the page with the vertical line through $D$ as its boundary. 
By construction, the cusps of $\Sigma$ arising from $A$ and $B$ cut $\partial B_\mathfrak{r}$ into two annuli. Further, from the geometry of $P_A$ and $P_B$, we see that each annulus has width $\sqrt[4]{2}$ and area $\sqrt{2}$.

The third cusp of $\Sigma$ cuts $\partial B_\mathfrak{b}$ into one annulus as shown in Figure \ref{fig:horo_blue}. To convince us of this diagram, note that $Q$ is the gluing of two ideal triangles: $DAC$, with edges $(2,2,0)$, and $CBD$, with edges $(1,1,0)$.
Figure \ref{fig:horo_blue} is the view from the point $D$ at infinity. If we look at the orbit of these triangles under $\pi_1(N)$,  we see that Figure \ref{fig:horo_blue} depicts exactly the hyperplanes in $\pi_1(N) \cdot P$ perpendicular to the page. Therefore, the annulus has width $1/\sqrt[4]{2}$ and area $2\sqrt{2}$ on $\partial B_\mathfrak{b}$.

The $\partial$-cusps of $M$ can be labeled $\{\mathfrak{r}_1, \mathfrak{r}_2, \mathfrak{b}'\}$, corresponding to the annuli above. From our analysis, $d_{\mathfrak{r}_i} = \sqrt[4]{2}$, $d_{\mathfrak{b}'}  = 1/\sqrt[4]{2}$, $\Vol(\mathfrak{r}_i) = \sqrt{2}/2$ and $\Vol(\mathfrak{b}') = \sqrt{2}$. Additionally, it is well known that $\Vol(M) = \Vol(N) = 4 \cdot \text{Catalan's constant} \approx 3.66386$. Theorem \ref{thm:bkfv} and the work of Masai-McShane \cite{Masai:2013ji} tells us that $$4 \cdot \text{Catalan's constant} = \sum_{\ell \in |\cO(M)|}  \frac{2\pi (\ell+1)}{e^{2\ell}-1} + 1 \left(\frac{\sqrt{2}}{2\left(\sqrt[4]{2}\right)^{2}}+\frac{\sqrt{2}}{2\left(\sqrt[4]{2}\right)^{2}}+\frac{\sqrt{2}}{\left(1/\sqrt[4]{2}\right)^{2}}\right).$$
Let $G(x)$ denote the inverse of $F_3(L)$, then, as $F_3$ is decreasing, we have
\[
\ell > G(4 \cdot \text{Catalan's constant} - 3) \approx 1.62629
\]
for all $\ell \in |\cO(M)|$.


\section{Identity for manifolds with cusped boundary}
In this section, we prove the extended Bridgeman-Kahn identity:
\bkifv*
Notice that $\Vol(B_\mathfrak{c})/d_\mathfrak{c}^{\,n-1}$ is independent of the choice of embedded neighborhood $B_\mathfrak{c}$. The asymptotics of our cusp coefficient are straightforward to analyze. In particular, one has
\begin{Prop} As $n \to \infty$, $$\frac{H(n-2) \; \Gamma\left(\frac{n-2}{2}\right)}{\sqrt{\pi} \; \Gamma\left(\frac{n-1}{2}\right)} = \sqrt{\frac{2}{\pi}}\left(\frac{\gamma}{\sqrt{n}} + \frac{\log(n)}{\sqrt{n}}\right) + O\left(\frac{1}{n^{3/2}}\right)$$
where $\gamma$ is Euler's constant.
\end{Prop}
\begin{proof} This observation follows directly of the well known asymptotic of $H(m)$ and $\Gamma(z)$. As $m,z \to \infty$, $$H(m) = \gamma + \log(m) + \frac{1}{2m} + O\left(\frac{1}{m^2}\right),$$
$$\frac{\Gamma(z+a)}{\Gamma(z+b)} = z^{a-b} \left(1 + \frac{(a-b)(a+b-1)}{2z} + O\left(\frac{1}{z^2}\right)\right)$$
where we take $z = n/2$, $a = -1$ and $b = -1/2$.
\end{proof}

\begin{Rem} (1) Observe that by $(ii)$ of Lemma \ref{thm:bko}, one has $$\lim_{l\to 0} l^{n-2}\, F_n(l)  = \frac{\pi^{\frac{n-2}{2}} \, H(n-2)\, \Gamma\left(\frac{n-2}{2}\right)}{\Gamma\left(\frac{n-1}{2}\right) \, \Gamma\left(\frac{n+1}{2}\right)} = \frac{\Vol(\bS^{n})}{2\pi} \frac{H(n-2) \; \Gamma\left(\frac{n-2}{2}\right)}{\sqrt{\pi} \; \Gamma\left(\frac{n-1}{2}\right)}.$$
Since both of these quantities compute volumes of tangent vectors, it is possible that there might be a direct relationship using a geometric rescaling argument. Unfortunately, our proof of Theorem \ref{thm:bkfv} does not provide such an insight.

(2) Additionally, $$\lim_{n \to 2} \; \frac{H(n-2) \; \Gamma\left(\frac{n-2}{2}\right)}{\sqrt{\pi} \; \Gamma\left(\frac{n-1}{2}\right)} = \frac{\pi}{3},$$ which is the cusp coefficient in the surface case obtained by Bridgeman  \cite{Bridgeman:2011ffa}.
\end{Rem}

We will start our proof by providing a full measure decomposition of the unit tangent bundle following the work of \cite{Bridgeman:2010fq} and \cite{Bridgeman:2011ffa}.  We then proceed by calculating the volume of each piece in the decomposition.  

\subsection{Decomposition of the Unit Tangent Bundle}

For the remainder of the article let $M$ be as in the statement of Theorem \ref{thm:bkfv}.
For each $v \in \UT M$, let $\exp_v\co I_v \to M$ be the maximal -- with respect to inclusion -- unit speed geodesic with $\exp_v'(0) = v$ and $I_v \subset \br$ an interval. 
Define $\ell_v$ to be the length of $\exp_v$. 
For each  $\al \in \cO(M)$, define
\[
V_\al  = \{ v \in  \UT M \mid \exp_v \text{ has finite length and } \exp_v \text{ is homotopic to  } \al \text{ relative to } \partial M\}.
\]
A universal covering argument shows that $\Vol(V_\al)$ only depends on the length of $\al$. 
The function relating $\Vol(V_\al)$ and the length of the orthogeodesic $\al$ is the $n^{th}$-Bridgeman-Kahn function:
\[
\Vol(\bS^{n-1})\cdot F_n(\ell(\al)) = \Vol(V_\al).
\]

To understand how much of the volume of $\UT M$ is covered by $\{ V_\al \mid \al \in \cO(M)\}$, we recall that the geodesic flow on a geodesically complete finite-volume hyperbolic manifold is ergodic \cite[Theorem 8.3.7]{Nicholls:1989ij}. In particular, ergodicity of the geodesic flow for the double of $M$ implies that $\exp_v$ must hit $\partial M$ in both directions for almost every $v\in \UT M$.
When $M$ is compact, $\partial M$ is closed and every finite-length $\exp_v$ is homotopic to a unique orthogeodisc. In this setting, it follows that $\bigcup_{\al \in \cO(M)} V_\al$ is full measure in $\UT M$.
To extend this construction to the case where $\partial M$ has a geometric structure with cusps, we must consider the volume of vectors that exponentiate to finite arcs homotopic out a $\partial$-cusp of $M$ relative $\partial M$. Notice, we do not worry about internal cusps of $M$ as the set of vectors that wander off into an internal cusp has measure zero by ergodicity.

For a $\partial$-cusp $\mathfrak{c}$ of $M$, let $$V_\mathfrak{c}  = \{ v \in  \UT M \mid \exp_v \text{ is homotopic out  } \mathfrak{c} \text{ relative } \partial M\}.$$  
Note that if $v\in V_\mathfrak{c}$, then $\exp_v$ has finite length.
It immediately follows that
\begin{equation}\label{eq:decomp}
\Vol(\UT M) = \sum_{\al \in \cO(M)} \Vol(V_\al) +  \sum_{\mathfrak{c} \in \mathfrak{B}} \Vol(V_\mathfrak{c})
\end{equation}

The quantity $\Vol(V_\al) = \Vol(\bS^{n-1}) \, F_n(\ell(\al))$ is completely determined by $\ell(\al)$. Hence, it is left for us to compute $\Vol(V_\mathfrak{c})$.

\subsection{Computing $\Vol(V_\mathfrak{c})$}

In this Section, we work in the upper half space model of $\bh^n$ to express $\Vol(V_\mathfrak{c})$ in integral form (see Equation \eqref{int:vol} below).  The integration will be performed in Secion \ref{sec:int}.

Let $\mathfrak{c}$ be a boundary cusp of $M$. There are exactly two boundary components $X_-$ and $X_+$ of $\partial M$ that meet every horoball neighborhood of $\mathfrak{c}$. Let $B_\mathfrak{c}$ be an embedded horoball neighborhood of $\mathfrak{c}$ in $M$ and let $d_\mathfrak{c}$ denote the Euclidean distance along $\partial B_\mathfrak{c}$ between $X_-$ and $X_+$. We fix a lift $\wt{B_\mathfrak{c}}$ of $B_\mathfrak{c}$ tangent to $\bndry \bh^n $at a point $p$. Such a choice determines unique lifts of $X_-$ and $X_+$ to complete hyperplanes $H_-$ and $H_+$ in $\bh^n$, respectively, that cobound $\wt{M}$ and satisfy $\ov{H}_- \cap \ov{H}_+ = p$.
Let $\Gamma_\mathfrak{c} < \pi_1(M)$ be the stabilizer of $p$.  Recall that $\Gamma_\mathfrak{c}$ is a discrete group of parabolic transformations.

Conjugating to take $p \mapsto \infty$, we can assume that every element $\gamma \in \Gamma_\mathfrak{c}$ acts on $\partial \bh^n = \text{span} \la e_1, \ldots, e_{n-1} \ra$ by $\gamma(x) = a_\gamma + A_\gamma x$, where $A_\gamma$ is an orthogonal transformation, $0\neq a_\gamma \in \partial \bh^n$, and $A_\gamma a_\gamma = a_\gamma$ \cite[Theorem 4.7.3]{Ratcliffe:2013gs}. 
Further, we can assume
\begin{equation*}\begin{split}\partial \wt{B}_\mathfrak{c} &= \{ \mathbf{x} \in \bh^n \mid x_n = 1 \}\\
H_- &= \{ \mathbf{x} \in \bh^n \mid x_1 = 0 \}\\
H_+ &= d_\mathfrak{c} e_1 + H_-.\end{split}\end{equation*}
 In particular, this implies that $a_\gamma \cdot e_1 = 0$ and $A_\gamma e_1 = e_1$ for all $\gamma \in \Gamma_\mathscr{v}$. Let
 \begin{equation*}\begin{split}
 V &= \{  \mathbf{x} \in \bh^n \mid 0 \leq x_1 \leq d_\mathfrak{c}\}
 \end{split}\end{equation*} denote the region between $H_-$ and $H_+$.
 Lastly, we also need to consider the two $\Gamma_\mathfrak{c}$-invariant subsets
 \begin{equation*}\begin{split}
 U_- &= \{ \mathbf{x} \in \ \partial \bh^n \mid x_1 < 0 \}\\
 U_+ &= \{ \mathbf{x} \in \partial \bh^n \mid x_1 > d_\mathfrak{c} \}.
 \end{split}\end{equation*}

Given $v \in V_\mathfrak{c}$, let $\wt\exp_v$ be a lift of $\exp_v$ to $\bh^n$ such that $\wt\exp_v$ is contained in $V$ with endpoints  on $H_-$ and $H_+$. 
If we take $\gamma$ to be the complete geodesic in $\bh^n$ containing $\wt\exp_v$, then $\gamma$ has one endpoint in $U_-$ and the other in $U_+$.  See Figure \ref{fig:setup} for a diagram of this setup.

Let $D$ be a fundamental domain for the action of $\Gamma_\mathfrak{c}$ on $U_-$, then 
\[
\Vol(V_\mathfrak{c}) = 2 \Vol \{ v \in \UT V \mid v \text{ is tangent to a complete geodesic going from } D \text{ to } U_+\}.
\]
For points $\mathbf{x}, \mathbf{y} \in \partial \bh^n$, let $\mathscr{G}(\mathbf{x}, \mathbf{y})$ be the complete hyperbolic geodesic connecting $\mathbf{x}$ and $\mathbf{y}$. Define $$L(\mathbf{x}, \mathbf{y}) = \text{hyperbolic length of }V \cap \mathscr{G}(\mathbf{x}, \mathbf{y}).$$ Note that $L(\mathbf{x}, \mathbf{y}) = \ell_v$ for every vector tangent to $\mathscr{G}(\mathbf{x}, \mathbf{y}) \cap V$.

\begin{figure}[t]
\begin{center}\begin{overpic}[scale=.5]{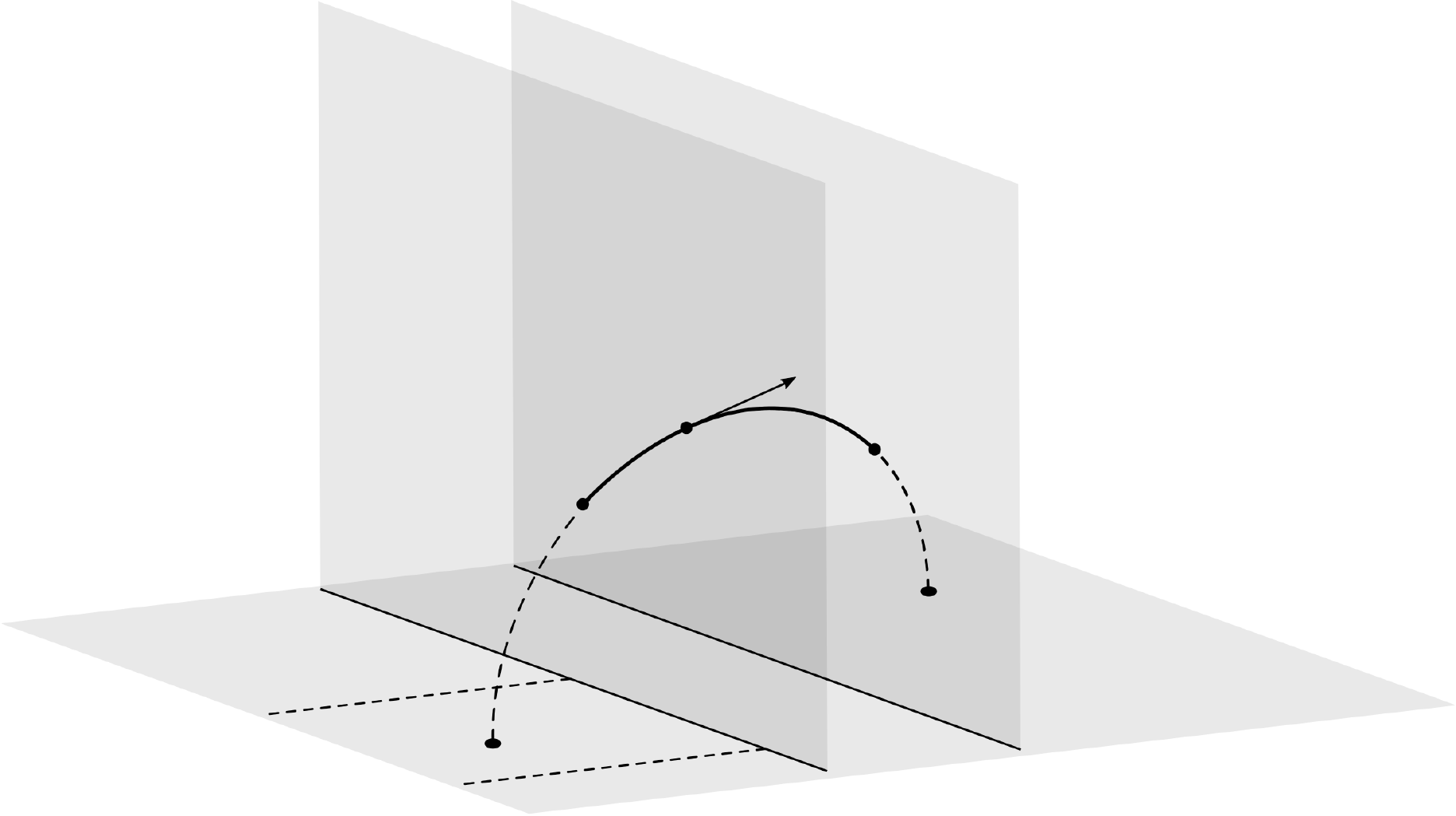}
\put(78,8.5){${U_+}$}
\put(17,10){${U_-}$}
\put(21,1){${D}$}
\put(61.5,35){${H_+}$}
\put(27,27){${H_-}$}
\put(35.5,4){$\mathbf{x}$}
\put(65.5,14.5){$\mathbf{y}$}
\put(45.5,28){$v$}
\end{overpic}
\end{center}
\caption{To compute $\Vol(V_\mathfrak{c})$, we must find the volume of all vectors $v \in T_1V$ for which the corresponding complete geodesic emanates from $D$ and terminates in $U_+$.}
\label{fig:setup}
\end{figure}
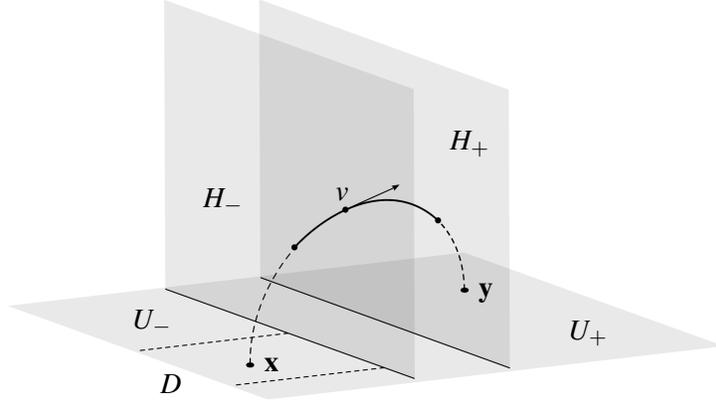

From Theorem \ref{thm:volume}, it follows that 
\begin{equation}\label{int:vol}
\Vol(V_\mathfrak{c})  = \int_{V_\mathfrak{c}} d \Omega = 2 \int_{\mathbf{y} \in U_+} \int_{\mathbf{x} \in D} \frac{2^{n-1}\,  L(\mathbf{x}, \mathbf{y}) \, d \mathbf{x} \, d \mathbf{y}}{| \mathbf{x} - \mathbf{y}|^{2n-2}}
\end{equation} 
where we integrate out the $dt$ to get $L(\mathbf{x}, \mathbf{y})$. As we see in Lemma \ref{lem:Lxy} below, $L(\mathbf{x},\mathbf{y})$ has a nice form as a function of $x_1$ and $y_1$.

\begin{Lem}\label{lem:Lxy} 
Let $U_\pm, L,$ and $d_\mathfrak{c}$ be as above.
Let $\mathbf{x} = \{x_1, \ldots, x_{n-1}\}\in U_-$ and  $\mathbf{y}=\{y_1, \ldots, y_{n-1}\} \in U_+$, then
\begin{equation}\label{eq:length}
L(\mathbf{x}, \mathbf{y}) = \frac{1}{2} \log \left( \frac{y_1(x_1-d_\mathfrak{c})}{x_1(y_1-d_\mathfrak{c})}\right).
\end{equation}
\end{Lem}

\begin{proof} The case when $n=2$ (see Figure \ref{fig:LxyH2}) is proven in \cite[Lemma 8]{Bridgeman:2007uh}.  We will use this fact in what follows.  

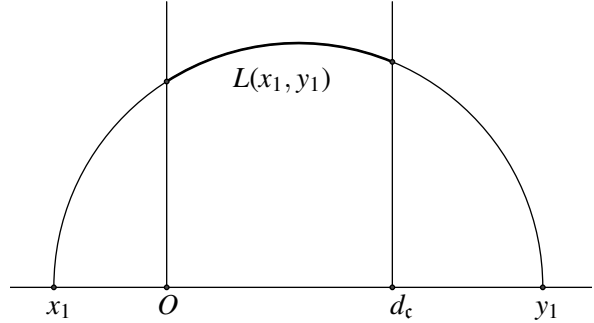
\begin{figure}[t]
\begin{center}\begin{overpic}[scale=.5]{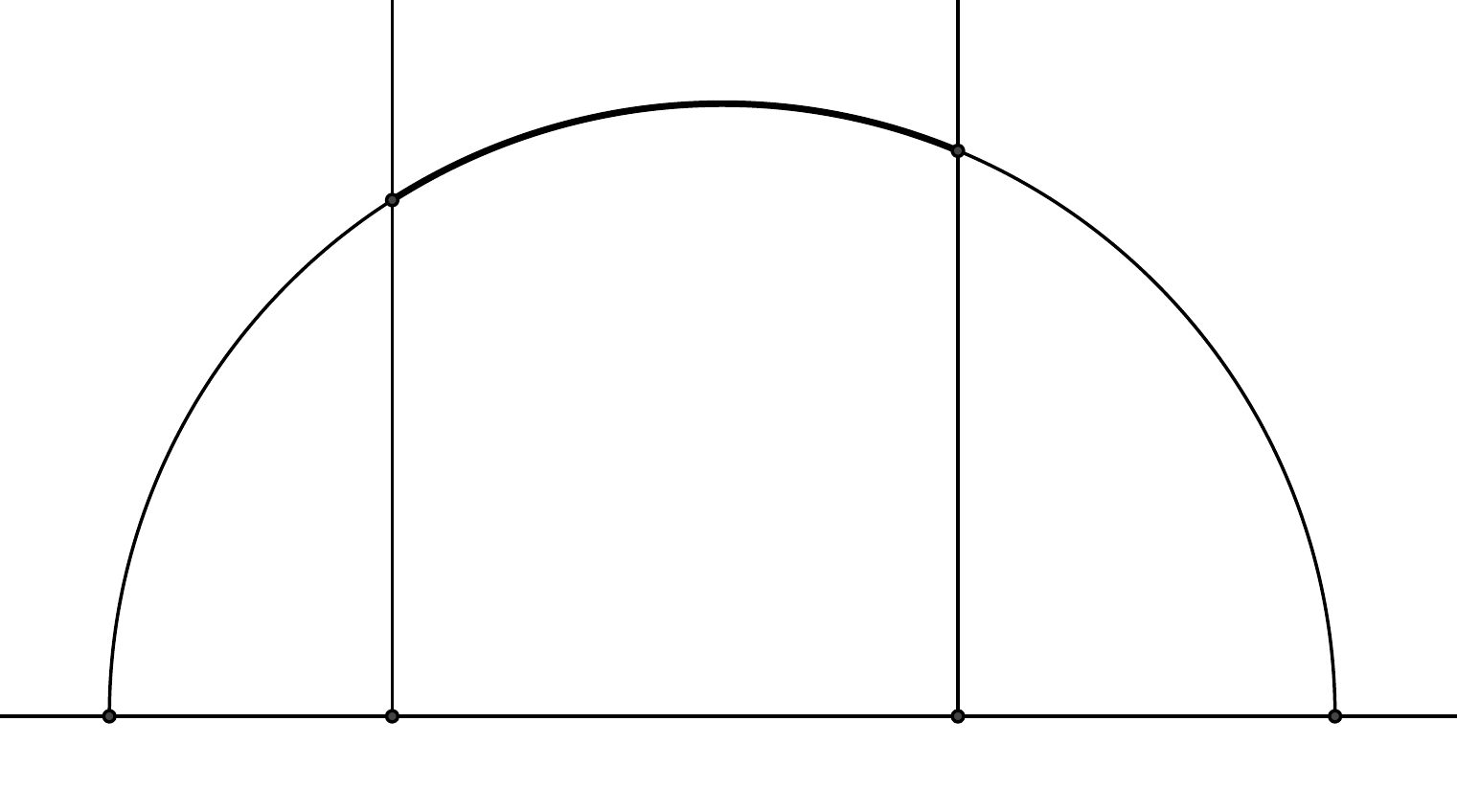}
\put(38,41){${L(x_1,y_1)}$}
\put(6,2){${x_1}$}
\put(90,2){${y_1}$}
\put(65,2){${d_\mathfrak{c}}$}
\put(25,2){$O$}
\end{overpic}
\end{center}
\vspace*{-2ex}
\caption{The diagram showing $L(x_1,y_1)$ in the plane $\bh_{\mathbf{x},\mathbf{y}}^2$ for Lemma \ref{lem:Lxy}.}
\label{fig:LxyH2}
\end{figure}

Without loss of generality, we may fix $\mathbf{x} = (x_1, 0, \ldots, 0)$ by applying parabolic transformations that fix $\infty$ and preserve $H_-, H_+$. We will show that $L(\mathbf{x}, \mathbf{y})$ depends only on $x_1, y_1$ and $d_\mathfrak{c}$. Consider Figure \ref{fig:len} showing $\mathbf{x}, \mathbf{y}$ on $\partial \bh^n$. Here, $\mathscr{G}(\mathbf{x}, \mathbf{y})$ is perpendicular to the page. There is a hyperbolic $2$-plane $\bh_{\mathbf{x},\mathbf{y}}^2$ transverse of $H_-$ in $\bh^n$ whose boundary is the line through $\mathbf{x},\mathbf{y}$. Now give $\partial \bh^2_{\mathbf{x},\mathbf{y}}$ coordinates by defining $0$ to be the point of intersection between $\partial H_-$ and $\partial \bh^2_{\mathbf{x},\mathbf{y}}$.  The coordinate of any other point $\mathbf{z}\in\partial \bh^2_{\mathbf{x},\mathbf{y}}$ can be obtained by calculating the Euclidean distance between $\bndry H_- \cap \partial \bh^2_{\mathbf{x},\mathbf{y}}$ and $\mathbf{z}$ in $\partial \bh^n$.

It follows that $L(\mathbf{x}, \mathbf{y})$ is the length of the arc on the geodesic $\mathscr{G}(-u,w+v)$ lying above the interval $(0,w)$ in  $\bh_{\mathbf{x},\mathbf{y}}^2$, where $u,v,$ and $w$ are as in Figure \ref{fig:len}. By construction, $w = \sec(\theta) \, d_\mathfrak{c}$, $u = \sec(\theta) \, | x_1|$ and $w+v = \sec(\theta)\, y_1$. Since multiplication by $\sec(\theta)$ is a hyperbolic isometry, we see that the length of the arc on the geodesic $\mathscr{G}(x_1, y_1)$ lying above the interval $(0, d_\mathfrak{c})$ in $\bh_{\mathbf{x},\mathbf{y}}^2$ has length $L(\mathbf{x}, \mathbf{y})$ (see Figure \ref{fig:LxyH2}).  
As we have reduced ourselves to the 2-dimensional case, we can invoke \cite[Lemma 8]{Bridgeman:2007uh} to obtain the desired result. 
\end{proof} 

\begin{figure}[h]
\begin{center}\begin{overpic}[scale=.7]{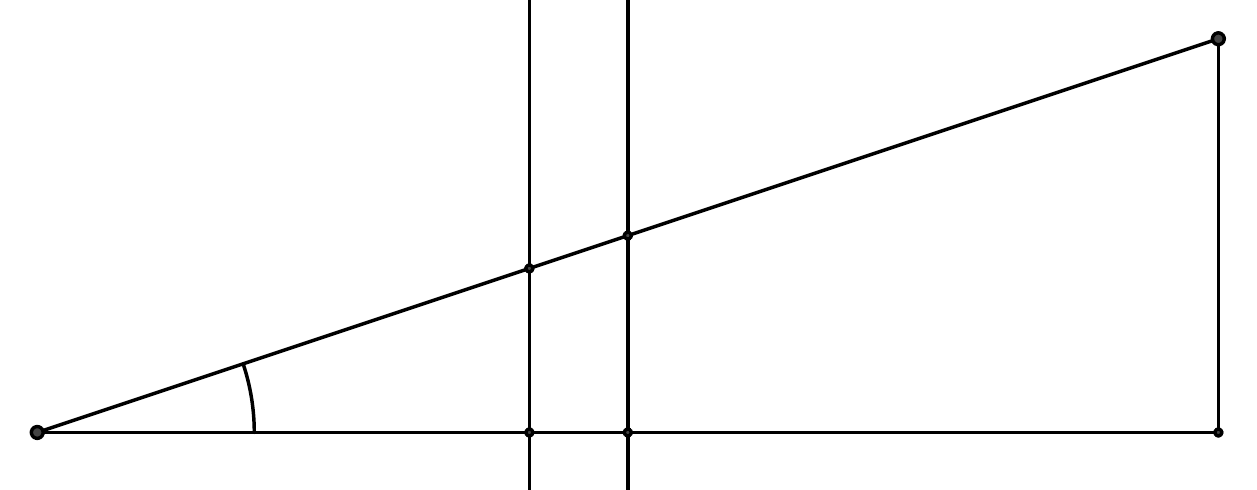}
\put(99,18){${|\mathbf{y} - y_1\mathbf{e}_1|}$}
\put(20,0){${|x_1|}$}
\put(22,13.5){${u}$}
\put(22,6.2){${\theta}$}
\put(66,0){${y_1 - d_\mathfrak{c}}$}
\put(70,29){${v}$}
\put(44.5,0){${d_\mathfrak{c}}$}
\put(38,1){${O}$}
\put(44.5,20.5){${w}$}
\put(52,35){${\partial H_+}$}
\put(28,35){${\partial H_-}$}
\put(115,35){${U_+}$}
\put(-15,35){${U_-}$}
\put(0,0){$\mathbf{x}$}
\put(96,38){$\mathbf{y}$}
\put(96,0){$\mathbf{y'}$}
\end{overpic}
\end{center}
\caption{The diagram above shows the points $\mathbf{x}, \mathbf{y}$ on $\partial \bh^n$ without $\infty$ in the $e_1, \ldots e_{n-1}$ coordinates. The point $O = (0, \ldots, 0)$ denotes the origin and horizontal is the $e_1$-axis.}
\label{fig:len}
\end{figure}
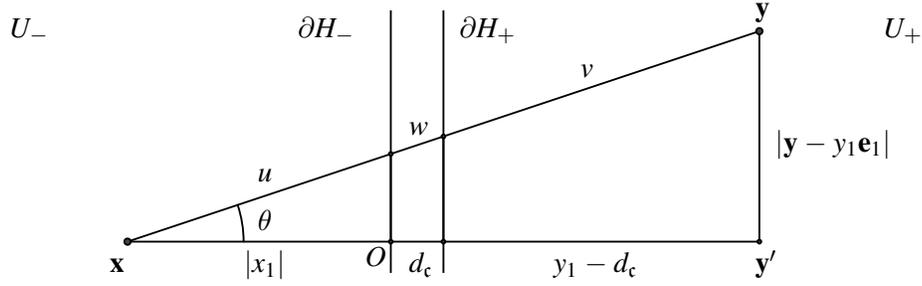

\subsection{Integration}\label{sec:int} 
To set up the integration, we observe that $D = (-\infty, 0) \times D'$ where $D'$ is a fundamental domain for the action of $\Gamma_\mathfrak{c}$ on $\partial H_- = \{ \mathbf{x} \in  \partial \bh^n \mid x_1 = 0\}$. Also, $U_+ = (d_\mathfrak{c}, \infty) \times \br^{n-2}$, refer once again to Figure \ref{fig:setup}.  Applying our observations to Equation (\ref{int:vol}) and making the substituions $w_i = y_i - x_i$ for $i = 2, \ldots n-1$, we obtain
\begin{equation} \label{int:y} \begin{split} \Vol(V_\mathfrak{c}) &=2^{n-1} \int_{-\infty}^0 \int_{d_\mathfrak{c}}^\infty \int_{D'} \int_{\br^{n-2}}  \frac{\log \left( \frac{y_1(x_1-d_\mathfrak{c})}{x_1(y_1-d_\mathfrak{c})}\right) \, d y_2 \ldots dy_{n-1} \, d x_2 \ldots x_{n-1} \, d y_1 \, d x_1}{\sqrt{(x_1-y_1)^2 + \sum_{i=2}^{n-1} (x_i - y_i)^2}^{\; 2n-2}}\\
& = 2^{n-1} \int_{-\infty}^0 \int_{d_\mathfrak{c}}^\infty \int_{D'} \int_{\br^{n-2}}  \frac{\log \left( \frac{y_1(x_1-d_\mathfrak{c})}{x_1(y_1-d_\mathfrak{c})}\right) \, d w_2 \ldots dw_{n-1} \, d x_2 \ldots x_{n-1} \, d y_1 \, d x_1}{\sqrt{(x_1-y_1)^2 + \sum_{i=2}^{n-1} w_i^2}^{\; 2n-2}}.
\end{split}\end{equation}

To integrate out $w_i$ for $i = 2, \ldots, n-1$, one can show with induction on $k \geq 3$ and the substitution $w = A \tan(\theta)$ that
\begin{equation}\label{int:trig}
\int_{-\infty}^\infty \frac{dw}{\sqrt{w^2+A^2}^{\; k}} = \frac{1}{A^{k-1}} \int_{-\pi/2}^{\pi/2} \cos^{k-2}(\theta) d \theta = \frac{\sqrt{\pi} \; \Gamma((k-1)/2)}{A^{k-1} \; \Gamma(k/2)}.
\end{equation}
The second equality following from the following calculation:

\begin{Lem}\label{prop:cosint} For $k \geq 3$,
\[
\int_{-\pi/2}^{\pi/2} \cos^{k-2}(\theta) d \theta = \frac{\sqrt{\pi} \; \Gamma((k-1)/2)}{\Gamma(k/2)}.
\] 
\end{Lem}

\begin{proof} We proceed by induction on $k$. For $k = 3$, we have $$ \int_{-\pi/2}^{\pi/2} \cos(\theta) d \theta = 2 = \frac{\sqrt{\pi}\cdot 1}{\sqrt{\pi}/2} = \frac{\sqrt{\pi} \; \Gamma((3-1)/2)}{\Gamma(3/2)}.$$ 
We also need to compute for $k = 4$, 
$$ \int_{-\pi/2}^{\pi/2} \cos^2(\theta) d \theta = \left[\frac{\theta}{2} + \frac{\sin(\theta)\cos{\theta}}{2} \right]_{-\pi/2}^{\pi/2} = \frac{\pi}{2} = \frac{\sqrt{\pi}(\sqrt{\pi}/2)}{1}= \frac{\sqrt{\pi} \; \Gamma((4-1)/2)}{\Gamma(4/2)}.$$ 

Using the induction assumption for $k > 4$, we have \begin{equation*} \begin{split}\int_{-\pi/2}^{\pi/2} \cos^{k-2}(\theta) d \theta &= \left[\frac{\cos^{k-3}(\theta)\sin(\theta)}{k-2} \right]_{-\pi/2}^{\pi/2} + \frac{k-3}{k-2} \int_{-\pi/2}^{\pi/2} \cos^{k-4}(\theta) d\theta \\
&= 0 + \frac{\sqrt{\pi} \; ((k-3)/2) \; \Gamma((k-3)/2)}{((k-2)/2) \; \Gamma((k-2)/2)} = \frac{\sqrt{\pi} \; \Gamma((k-1)/2)}{\Gamma(k/2)}.
\end{split}\end{equation*}
\end{proof}

Now, for $w_i$ with $i \geq 2$, we let $A_i = \sqrt{(x_1-y_1)^2 + \sum_{j = i+1}^{n-1} w_j^2}$ and $k = 2n-i$. Applying equation (\ref{int:trig}) recursively for $i \geq 2$, we obtain 
\begin{equation*} 
\begin{split} 
\Vol(V_\mathfrak{c}) &=\frac{2^{n-1} \pi^{(n-2)/2}\; \Gamma(n/2)}{\Gamma(n-1)} \int_{-\infty}^0 \int_{d_\mathfrak{c}}^\infty \int_{D'} \frac{\log \left( \frac{y_1(x_1-d_\mathfrak{c})}{x_1(y_1-d_\mathfrak{c})}\right) d x_2 \ldots x_{n-1} \, d y_1 \, d x_1}{(y_1-x_1)^n}\\
&=\frac{2^{n-1} \pi^{(n-2)/2} \Vol_\mathbb{E}(D') \; \Gamma(n/2)}{\Gamma(n-1)} \int_{-\infty}^0 \int_{d_\mathfrak{c}}^\infty \frac{\log \left( \frac{y_1(x_1-d_\mathfrak{c})}{x_1(y_1-d_\mathfrak{c})}\right) d y_1 \, d x_1}{(y_1-x_1)^n},
\end{split}\end{equation*}
where $\Vol_\mathbb{E}(D')$ is the Euclidean volume of $D'$.  Note that $\Vol_\mathbb{E}(D')$ is finite.  Indeed, the fundamental domain for the action of $\Gamma_\mathfrak{c}$ on $\{x \in \bh^n : x_n > 1\}$ is parametrized as $[0,d_\mathfrak{c}] \times D' \times (1, \infty)$.  
A standard calculation yields
\[
\Vol_\mathbb{E}\left(D'\right) = \frac{(n-1) \Vol(B_\mathfrak{c})}{d_\mathfrak{c}}.
\]

For the remaining integral, we turn to the following lemma, whose proof is temporarily postponed.
\begin{Lem}\label{int:main}  For $n \geq 3$
$$\int_{-\infty}^0 \int_{d_\mathfrak{c}}^\infty \frac{\log \left( \frac{y(x-d_\mathfrak{c})}{x(y-d_\mathfrak{c})}\right) d y \, d x}{(y-x)^n} = \frac{2 \, H(n-2)}{(n-1)(n-2)\, d_\mathfrak{c}^{\,n-2}}$$
\end{Lem}
It then follows from Lemma \ref{int:main}  that 
\begin{equation*} \Vol(V_\mathfrak{c}) = \frac{2^n \pi^{(n-2)/2} \; H(n-2) \; \Gamma(n/2)}{(n-2)\; \Gamma(n-1)} \frac{\Vol(B_\mathfrak{c})}{d_\mathfrak{c}^{\,n-1}}.\end{equation*}
In our setup so far, we have been calculating volume in the unit tangent bundle.  To calculate the volume of $M$, it is necessary to divide $\Vol(V_\mathfrak{c})$ by $\Vol(\mathbb{S}^{n-1})$ to get
\begin{equation*} \frac{1}{\Vol(\bS^{n-1})} \Vol(V_\mathfrak{c}) = \frac{2^{n-1}\; H(n-2) \; \Gamma(n/2)^2}{\pi  \, (n-2)\; \Gamma(n-1)} \frac{\Vol(B_\mathfrak{c})}{d_\mathfrak{c}^{\,n-1}}.\end{equation*}
By the duplication formula for $\Gamma(\cdot)$, one has $$2^{1-(n-1)}\sqrt{\pi}\,\Gamma(n-1) = \Gamma\left(\frac{n-1}{2}\right) \Gamma\left(\frac{n-1}{2} + \frac{1}{2} \right) = \Gamma\left(\frac{n-1}{2}\right) \Gamma\left(\frac{n}{2}\right).$$
Using this relation, we can simplify
\begin{equation}\begin{split} \frac{1}{\Vol(\bS^{n-1})} \Vol(V_\mathfrak{c}) &=  \frac{2 \; H(n-2) \; \Gamma(\frac{n}{2})}{\sqrt{\pi}  \, (n-2)\; \Gamma(\frac{n-1}{2})} \frac{\Vol(B_\mathfrak{c})}{d_\mathfrak{c}^{\,n-1}}\\
&= \frac{2 \; H(n-2) \, \left(\frac{n}{2}-1\right) \Gamma(\frac{n-2}{2})}{\sqrt{\pi}  \, (n-2)\; \Gamma(\frac{n-1}{2})} \frac{\Vol(B_\mathfrak{c})}{d_\mathfrak{c}^{\,n-1}}\\
&= \frac{H(n-2) \; \Gamma(\frac{n-2}{2})}{\sqrt{\pi}  \; \Gamma(\frac{n-1}{2})} \frac{\Vol(B_\mathfrak{c})}{d_\mathfrak{c}^{\,n-1}}.\\
\end{split}\end{equation}
Up to the proof of Lemma \ref{int:main}, our version of the Bridgeman-Kahn identity is complete by assembling our computations and the decomposition in Equation (\ref{eq:decomp}).
 $$\boxed{\Vol(M) = \sum_{\ell \in |\cO(M)|} F_n(\ell) + \frac{H(n-2) \; \Gamma\left(\frac{n-2}{2}\right)}{\sqrt{\pi} \; \Gamma\left(\frac{n-1}{2}\right)} \sum_{\mathfrak{c} \in \mathfrak{B}}  \frac{\Vol(B_\mathfrak{c})}{d_\mathfrak{c}^{\,n-1}}}$$

\begin{proof}[Proof of Lemma \ref{int:main}] We first split up the integral into three pieces $$I = \int_{-\infty}^0 \int_{d_\mathfrak{c}}^\infty \frac{\log \left( \frac{y(x-d_\mathfrak{c})}{x(y-d_\mathfrak{c})}\right) d y \, d x}{(y-x)^n} = I_1 - I_2 - I_3$$ where
\begin{equation*}\begin{split}
I_1 &= \int_{-\infty}^0 \int_{d_\mathfrak{c}}^\infty \frac{\log \left(d_\mathfrak{c} - x\right) d y \, d x}{(y-x)^n} = \frac{(n-2)\log(d_\mathfrak{c})+1}{(n-1)(n-2)^2 \, d_\mathfrak{c}^{\,n-2}}\\
I_2 &= \int_{-\infty}^0 \int_{d_\mathfrak{c}}^\infty \frac{\log \left(-x/y\right) d y \, d x}{(y-x)^n} = \frac{-H(n-2)}{(n-1)(n-2) \, d_\mathfrak{c}^{\,n-2}}\\
I_3 &= \int_{-\infty}^0 \int_{d_\mathfrak{c}}^\infty \frac{\log \left(y-d_\mathfrak{c}\right) d y \, d x}{(y-x)^n} = \frac{\log \left(d_\mathfrak{c}\right) - H(n-3)}{(n-1)(n-2) \, d_\mathfrak{c}^{\,n-2}}
\end{split}\end{equation*}
The equalities on the right come from Lemmas \ref{I1}, \ref{I2} and \ref{I3} below. Combining, we have our desired result:
\begin{equation*}\begin{split}
I &= I_1 - I_2 - I_3\\
&=\frac{(n-2)\log(d_\mathfrak{c})+1}{(n-1)(n-2)^2 \, d_\mathfrak{c}^{\,n-2}} +\frac{H(n-2)}{(n-1)(n-2) \, d_\mathfrak{c}^{\,n-2}} + \frac{H(n-3) - \log \left(d_\mathfrak{c}\right)}{(n-1)(n-2) \, d_\mathfrak{c}^{\,n-2}}\\
&=\frac{1}{(n-1)(n-2) \, d_\mathfrak{c}^{\,n-2}}\left(\frac{1}{n-2}+ H(n-2) + H(n-3) \right)\\
&= \frac{2\, H(n-2)}{(n-1)(n-2)\, d_\mathfrak{c}^{\,n-2}}.
\end{split}\end{equation*}
\end{proof}

\begin{Lem}\label{I1}
\[
I_1 = \frac{(n-2)\log(d_\mathfrak{c})+1}{(n-1)(n-2)^2 \, d_\mathfrak{c}^{\,n-2}}
\]
\end{Lem}

\begin{proof}
\begin{equation*}\begin{split}
I_1 &= \int_{-\infty}^0 \int_{d_\mathfrak{c}}^\infty \frac{\log \left(d_\mathfrak{c} - x\right) d y \, d x}{(y-x)^n}\\
&= \frac{1}{n-1} \int_{-\infty}^0 \frac{\log \left(d_\mathfrak{c}-x\right) d x}{(d_\mathfrak{c}-x)^{n-1}}\\
&= \frac{1}{n-1}\left[ \frac{\log(d_\mathfrak{c}-x)}{(n-2)(d_\mathfrak{c}-x)^{n-2}} + \frac{1}{(n-2)^2(d_\mathfrak{c}-x)^{n-2}} \right]_{-\infty}^0\\ 
&= 
{\frac{(n-2)\log(d_\mathfrak{c})+1}{(n-1)(n-2)^2 \, d_\mathfrak{c}^{\,n-2}}}
\end{split}\end{equation*}
\end{proof}

\begin{Lem}\label{I2}
\[
I_2 = \frac{-H(n-2)}{(n-1)(n-2) \, d_\mathfrak{c}^{\,n-2}}
\]
\end{Lem}

\begin{proof}
We first use the change of coordinates $z = x/y$ and $y = y$, where $d y \, d x = y dy \, dz$. With the proper change in the limits of integration,

\begin{equation*}\begin{split}
I_2 &= \int_{-\infty}^0 \int_{d_\mathfrak{c}}^\infty \frac{\log \left(-x/y\right) d y \, d x}{(y-x)^n}\\
&= \int_{-\infty}^0 \int_{d_\mathfrak{c}}^\infty \frac{\log \left(-z\right) d y \, d z}{y^{n-1}\,(1-z)^n}\\
&= \frac{1}{(n-2)  \, d_\mathfrak{c}^{\,n-2}} \int_{-\infty}^0 \frac{\log \left(-z\right) d z}{(1-z)^n}
\end{split}\end{equation*}
Next, we change coordinates to $w = 1/(1-z)$ with $dw = dz/(1-z)^2$, giving \begin{equation*}\begin{split}
I_2 &= \frac{1}{(n-2)  \, d_\mathfrak{c}^{\,n-2}} \int_0^1 \log \left(\frac{1}{w}-1\right)w^{n-2} d w
\end{split}\end{equation*}

Let $m=n-2$.  We continue by splitting this integral into two parts,
\begin{equation*}\begin{split}
\int_0^1 \log \left(\frac{1}{w}-1\right)w^m d w &= \int_0^1 \left[\log \left(1-w\right)w^m -  \log \left(w \right)w^m\right] d w\\
& = \int_0^1 \frac{\log(1-w)}{-m-1} d \left(1-w^{m+1}\right) - \int_0^1 \frac{\log(w)}{m+1} d \left(w^{m+1}\right)
\end{split}\end{equation*}
As $m \geq 0$, the two integrals inside are as follows
 \begin{equation*}\begin{split}
\int_0^1 \frac{\log(1-w)}{-m-1} d \left(1-w^{m+1}\right) &= \left[\frac{\log(1-w)\left(1-w^{m+1}\right)}{-m-1}\right]_0^1 - \frac{1}{m+1} \int_0^1 \frac{1-w^{m+1}}{1-w} dw\\
&= 0 - \frac{H(m+1)}{m+1}.
\end{split}\end{equation*}
and
 \begin{equation*}\begin{split}
\int_0^1 \frac{\log(w)}{m+1} d \left(w^{m+1}\right) &= \left[\frac{\log(w) \, w^{m+1}}{m+1}\right]_0^1 - \frac{1}{m+1} \int_0^1 w^m dw\\
&= 0 - \frac{1}{(m+1)^2}
\end{split}\end{equation*}
Combining, we see that 
\begin{equation}\label{eq:harmonic}
\int_0^1 \log \left(\frac{1}{w}-1\right)w^{m} d w= \frac{1}{m+1}\left( - H(m+1) + \frac{1}{m+1}\right) = \frac{-H(m)}{m+1}
\end{equation}
\end{proof}

\begin{Lem}\label{I3}
\[
I_3 = \frac{\log \left(d_\mathfrak{c}\right) - H(n-3)}{(n-1)(n-2) \, d_\mathfrak{c}^{\,n-2}}
\]
\end{Lem}

\begin{proof}
Using the change of coordinates $u = y/d_\mathfrak{c}$ and $z = x/y$, so that $dy \, dx = u \, d_\mathfrak{c}^2 \, du \, dz$, we have
\begin{equation*}\begin{split}
I_3 &= \int_{-\infty}^0 \int_{d_\mathfrak{c}}^\infty \frac{\log \left(y-d_\mathfrak{c}\right) d y \, d x}{(y-x)^n}\\
&=  \int_{-\infty}^0 \int_{1}^\infty \frac{\log \left(d_\mathfrak{c}(u-1)\right) d u \, d z}{u^{n-1} (1-z)^n \, d_\mathfrak{c}^{\,n-2}}\\
&= \frac{1}{d_\mathfrak{c}^{\,n-2}} \left(\int_{-\infty}^0 \frac{\log \left(d_\mathfrak{c}\right) d z}{(1-z)^n} \int_{1}^\infty \frac{d u}{u^{n-1}} + \int_{-\infty}^0 \frac{d z}{(1-z)^n} \int_{1}^\infty \frac{\log \left(u-1\right) d u}{u^{n-1}}\right)\\
&= \frac{1}{d_\mathfrak{c}^{\,n-2}} \left(\frac{\log \left(d_\mathfrak{c}\right)}{n-1} \frac{1}{n-2}+ \frac{1}{n-1} \int_{1}^\infty \frac{\log \left(u-1\right) d u}{u^{n-1}}\right).
\end{split}\end{equation*}
We do one last change of coordinates to $w = 1/u$ with $dw = -du/u^2$ and reuse Equation \eqref{eq:harmonic} to obtain
\begin{equation*}\begin{split}
I_3 & =  \frac{1}{(n-1)\, d_\mathfrak{c}^{\,n-2}} \left(\frac{\log \left(d_\mathfrak{c}\right)}{n-2} + \int_{0}^1 \log \left(\frac{1}{w}-1\right) w^{n-3} dw\right)\\
&=   \frac{1}{(n-1)\, d_\mathfrak{c}^{\,n-2}} \left(\frac{\log \left(d_\mathfrak{c}\right)}{n-2} - \frac{H(n-3)}{n-2}\right)\\
&={\frac{\log \left(d_\mathfrak{c}\right) - H(n-3)}{(n-1)(n-2) \, d_\mathfrak{c}^{\,n-2}}}
\end{split}\end{equation*}
\end{proof}

\footnotesize 
\bibliographystyle{amsalpha}	
\bibliography{references}

\end{document}